\documentclass[12pt]{amsart}
\usepackage[cp1251]{inputenc}
\usepackage[T2A]{fontenc}
\usepackage[english]{babel}
\usepackage{amsmath,amsfonts,amssymb}
\usepackage{geometry}
\usepackage{amsmath, amsthm, amscd, amsfonts, amssymb, graphicx, color}
\usepackage[bookmarksnumbered, colorlinks, plainpages]{hyperref}

\numberwithin{equation}{section}
\newtheorem{theorem}{Theorem}[section]
\newtheorem{lemma}[theorem]{Lemma}

\theoremstyle{remark}
\newtheorem{remark}[theorem]{Remark}

\theoremstyle{definition}

\begin{document}

\title[Approximation properties of multipoint boundary problems]{Approximation properties\\ of multipoint boundary-value problems}


\author[H. Masliuk]{Hanna Masliuk}
\address{National Technical University of Ukraine “Igor Sikorsky Kyiv Polytechnic Institute”, Peremohy Avenue 37, 03056, Kyiv-56, Ukraine}
\email{masliukgo@ukr.net}


\author[O. Pelekhata]{Olha Pelekhata}
\address{National Technical University of Ukraine “Igor Sikorsky Kyiv Polytechnic Institute”, Peremohy Avenue 37, 03056, Kyiv-56, Ukraine}
\email{o.pelehata-2017@kpi.ua}


\author[V. Soldatov]{Vitalii Soldatov}
\address{Institute of Mathematics, National Academy of Sciences of Ukraine, Tereshchenkivska Str. 3, 01004 Kyiv-4, Ukraine}
\email{soldatovvo@ukr.net, soldatov@imath.kiev.ua}



\dedicatory{To Vladimir Andreevich Mikhailets on the occasion of his 70th birthday}


\begin{abstract}
We consider a wide class of linear boundary-value problems for systems of $r$-th order ordinary differential equations whose solutions range over the normed complex space $(C^{(n)})^m$ of $n\geq r$ times continuously differentiable functions $y:[a,b]\to\mathbb{C}^{m}$. The boundary conditions for these problems are of the most general form $By=q$, where $B$ is an arbitrary continuous linear operator from $(C^{(n)})^{m}$ to $\mathbb{C}^{rm}$. We prove that the solutions to the considered problems  can be approximated in $(C^{(n)})^m$ by solutions to some multipoint boundary-value problems. The latter problems do not depend on the right-hand sides of the considered problem and are built explicitly.
\end{abstract}

\maketitle

\section{Introduction}\label{sec_Int}

Questions about justification of passage to the limit in differential systems arise in various mathematical and applied problems. These questions are best investigated for the Cauchy problem for systems of first-order ordinary differential equations (ODEs); see, e.g. \cite{Gikhman1952, KrasnoselskiiKrein1955, KurzweilVorel1957CMJ, Levin1967dan, Opial1967, Reid1967}. Parameter-dependent boundary-value problems are far less investigated than the Cauchy problem due to the  wide variety of boundary conditions. Pioneer results on this topic are obtained by Kiguradze \cite{Kiguradze1987, Kiguradze1975, Kiguradze2003} and Ashordia \cite{Ashordia1996}, who introduced and investigated a class of general linear boundary-value problems for systems of first-order ODEs. Kiguradze and Ashordia found sufficient conditions under which the solutions to these problems are continuous with respect to the parameter in the normed space $C([a,b],\mathbb R^m)$, with $m$ being the number of equations in the system. Recently Mikhailets and his disciples \cite{GnypKodlyukMikhailets2015UMJ, Gnyp2016, GnypMikhailetsMurach2017, KodliukMikhailets2013JMS, MikhailetsMurachSoldatov2016MFAT, MikhailetsMurachSoldatov2016EJQTDE, MikhailetsReva2008DAN8, Soldatov2015UMJ} introduced and investigated the broadest classes of linear boundary-value problems for systems of $r$-th order ODEs whose solutions range over chosen normed function spaces such as spaces of $n\geq r$ times continuously differentiable functions, H\"older spaces, Sobolev spaces and some others. Boundary conditions for these problems are posed in the most general form $By=q$ where $B$ is an arbitrary linear continuous operator acting from the chosen space to $\mathbb{C}^{rm}$. Such conditions cover all the known types of boundary conditions. These problems are called generic with respect to the chosen space. Constructive criteria for the solutions of these problems to be continuous in the chosen space with respect to the parameter are found.

Among important examples of the above-mentioned problems are multipoint boundary-value problems. The latter are simpler with respect to the structure of the boundary operator and to the finding of their solutions as compared with the problems that contain integral boundary conditions, for example. We therefore ask whether it is possible to approximate solutions of the above-mentioned boundary problems with solutions to some multipoint boundary problems. The purpose of this paper is to substantiate a positive answer to this question for  boundary problems that are generic with respect to the space $(C^{(n)})^{m}:=C^{(n)}([a,b],\mathbb C^{m})$ where $n\geq r$. Of course, this question relates closely to the interpolation of functions.
Our reasoning relies on the fact that the linear span of all continuous linear functionals on $C([a,b])$ which correspond to Dirac measures supported at single points is sequentially dense in the dual of $C([a,b])$ with respect to the $w^*$-topology. We will also give a constructive proof of this fact.


\section{Main results}

Let $a,b\in\mathbb{R}$ satisfy $a<b$. As usual, $C^{(l)}:=C^{(l)}([a,b],\mathbb{C})$, with $l\in\mathbb{N}$, 
 stands
for the Banach space of $l$ times continuously differentiable complex-valued functions on $[a,b]$. This space is endowed with the norm
\begin{equation*}
\|x\|_{(l)}:=\sum_{j=0}^{l}\max_{a\leq t\leq b}|x^{(j)}(t)|
\end{equation*}
and induces the Banach spaces
\begin{equation*}
(C^{(l)})^{m}:=C^{(l)}([a,b],\mathbb{C}^{m})\quad\mbox{and}\quad
(C^{(l)})^{m\times m}:=C^{(l)}([a,b],\mathbb{C}^{m\times m}).
\end{equation*}
They consist of all vector-valued functions or $(m\times m)$-matrix-valued functions whose components belong to $C^{(l)}$. The norms in these spaces are equal to the sum of the norms in $C^{(l)}$ of all components of the vector-valued or matrix-valued functions and will be denoted by $\|\cdot\|_{(l)}$ as well.

We arbitrarily choose integers $n$, $r$, and $m$ such that $n\geq r\geq1$ and $m\geq1$. On $[a,b]$, we consider a boundary-value problem of the form
\begin{align}\label{2MSp.syste}
y^{(r)}(t)+\sum_{l=1}^r A_{r-l}(t)y^{(r-l)}(t)&=f(t)
\quad\mbox{whenever}\quad a\leq t\leq b,\\
By&=q. \label{2MSp.kue}
\end{align}
Here, the unknown vector-valued function $y$ is supposed to be in $(C^{(n)})^{m}$, whereas the matrix-valued functions $A_{r-l}\in (C^{(n-r)})^{m\times m}$, with $1\leq l\leq r$, vector-valued function $f\in (C^{(n-r)})^{m}$, vector $q\in \mathbb C^{rm}$, and continuous linear operator
\begin{equation}\label{2MSp.oper}
B:(C^{(n)})^{m}\rightarrow \mathbb C^{rm}
\end{equation}
are arbitrarily given. Following \cite{MikhailetsMurachSoldatov2016MFAT}, we say that the boundary-value problem \eqref{2MSp.syste}, \eqref{2MSp.kue} is generic with respect to the space $C^{(n)}$.

We suppose this problem to have a unique solution $y\in(C^{(n)})^m$ for all $f\in (C^{(n-r)})^{m}$ and $q\in \mathbb C^{rm}$. According to \cite[Theorem~2]{Soldatov2015UMJ}, this is equivalent to the fact that the corresponding homogeneous problem (with $f(\cdot)\equiv0$ and $q=0$) has only the zero solution $y(\cdot)\equiv0$.

Let $\mathcal{X}$ be a dense subset of $(C^{(n-r)})^{m\times m}$.
Consider a sequence of multipoint bound\-ary-value problems of the form
\begin{align}\label{3MSp.syste}
L_{k}y_k(t):=y_k^{(r)}(t)+\sum_{l=1}^r A_{r-l,k}(t)y^{(r-l)}_{k}(t)&=f(t)
\quad\mbox{whenever}\quad a\leq t\leq b,
\\B_{k}y_k:=\sum\limits_{j=1}^{p_k}\sum\limits_{l=0}^{n}
\beta_k^{j,l}\,y^{(l)}_k(t_{k,j})&=q. \label{MSp.2B_m}
\end{align}
They are parameterized with the integer $k\geq 1$, and their right-hand sides are the same as those of the problem \eqref{2MSp.syste}, \eqref{2MSp.kue}. Here, each $A_{r-l,k}\in \mathcal{X}$. Besides, we suppose that $p_k\in \mathbb{N}$ 
 and $\beta_k^{j,l}\in \mathbb{C}^{rm \times m}$ and $t_{k,j} \in [a,b]$ for all admissible values of $k$, $j$, and~$l$. The solutions $y_k$ are considered in the class $(C^{(n)})^{m}$. Evidently, the linear operator $B_k$ acts continuously from  $(C^{(n)})^{m}$ to $\mathbb C^{rm}$.  Thus, every problem \eqref{3MSp.syste}, \eqref{MSp.2B_m} is also generic with respect to the space $C^{(n)}$.

\begin{theorem}\label{MSp.2apr_thm}
For the boundary-value problem \eqref{2MSp.syste}, \eqref{2MSp.kue},  there exists a sequence of multipoint boundary-value problems of the form \eqref{3MSp.syste}, \eqref{MSp.2B_m} such that they are uniquely solvable whenever $k\gg1$ and that
\begin{equation}\label{MSp.2granum}
y_{k}\to y \;\;\mbox{in}\;\; (C^{(n)})^m\;\;\mbox{as}\;\;k\rightarrow\infty.
\end{equation}
This sequence can be chosen not depending on $f$ and $q$ and can be  built explicitly.
\end{theorem}

\section{Preliminaries}

Our proof of Theorem \ref{MSp.2apr_thm} uses two results, which are of independent interest and will be given in this section.

We let $\mathrm{NBV}:=\mathrm{NBV}([a,b],\mathbb{C})$ denote the complex linear space of those functions $g:[a,b]\to\mathbb{C}$ that are of bounded variation on $[a,b]$, are left-continuous on $(a,b)$, and satisfy $g(a)=0$. This space is endowed with the norm which is equal to the total variation $\operatorname{V}(g,[a,b])$ of $g$ on $[a,b]$. The space $\mathrm{NBV}$ is complete, nonseparable, and nonreflexive with respect to this norm \cite[Chapter~IV, Sections 12 and 15]{DanfordShvarts1958}.

According to F.~Riesz's theorem, every continuous linear functional $\ell$ on the complex Banach space $C=C^{(0)}([a,b],\mathbb{C})$ admits the unique representation
\begin{equation}\label{l}
\ell(x) = \int\limits_a^b x(t)dg(t)\quad\mbox{whenever}\quad x\in C
\end{equation}
for some function $g\in \mathrm{NBV}$. Moreover, the corresponding mapping  $\mathcal{I}:g\mapsto \ell$ sets an isometric isomorphism $\mathcal{I}\colon\mathrm{NBV}\leftrightarrow C^\prime$ \cite[Chapter~IV, Section~13, Exercise~35]{DanfordShvarts1958}. Here, $C^\prime$ denotes the dual of $C$, we interpreting  $C^\prime$ as a Banach space endowed with the norm of functionals. The integral in \eqref{l} is understood in the Riemann--Stieltjes sense or as the Lebesgue--Stieltjes integral with respect to the complex-valued measure generated by $g$.

The space $\mathrm{NBV}$ is isometrically isomorphic to the Banach space of all complex-valued measures on $[a,b]$ (the norm in the latter space is the total variation of the measure) \cite[Chapter~IV, Section~12]{DanfordShvarts1958}. The Dirac measures supported at single points are the simplest examples of these measures. We let  $\chi_{\gamma}$ denote the characteristic function (indicator) of a subset $\gamma$ of $[a,b]$ and note that the characteristic functions
$\chi_{(c,b]}$, with $a\leq c\leq b$, and $\chi_{\{b\}}$ correspond to such Dirac measures under this isometric isomorphism.
Let $S:=S([a,b],\mathbb{C})$ denote the complex linear span of these characteristic functions. The set $\mathcal{I}(S)$ is dense in $C^\prime$ with respect to the $w^*$-topology \cite[Section IV.5, p.~114]{ReedSimon}. However, this topology is not metrizable. Hence, the indicated density does not imply that $\mathcal{I}(S)$ is sequentially dense in $C^\prime$ with respect to this topology. In fact, the latter density holds true, which follows from the next two results.

\begin{theorem}\label{th_approx}
Let $g\in
\mathrm{NBV}$. Then there exists a sequence $(g_k)^\infty_{k=1}$ of functions $g_k\in S$ such that
\begin{equation}\label{MSp.granum}
\sup_{a\leq t \leq b}|g_{k}(t)-g(t)|\rightarrow0 \quad \mbox{as}\quad k\to\infty,
\end{equation}
and
\begin{equation}\label{MSp.ii}
\sup_{k\geq1}\operatorname{V}(g_{k},[a,b])<\infty.
\end{equation}
\end{theorem}

Theorem~\ref{th_approx} and Helly's theorem \cite[Chapter~IV, Section~16, p.~391]{DanfordShvarts1958} immediately yield the following result:

\begin{theorem}\label{th_dense}
The set $\mathcal{I}(S)$ is sequentially dense in the space $C^\prime$ with respect to the $w^*$-topology; i.e.,
for every functional  $\ell \in C^\prime$ there exists a sequence $(g_k)^\infty_{k=1}\subset S$ such that
\begin{equation}\label{4aut.t3}
   \lim_{k\to \infty}\int\limits_a^b x(t)dg_k(t)= \ell(x) \quad \mbox{for every} \quad x\in C.
\end{equation}
\end{theorem}

\noindent\emph{{Proof of Theorem~\textup{\ref{th_approx}}.}}
Every function $g\in\mathrm{NBV}$ is a linear combination of two real-valued functions belonging to $\mathrm{NBV}$, whereas every real-valued function from $\mathrm{NBV}$ is the difference of two increasing functions belonging to $\mathrm{NBV}$. (We interpret increasing of functions in the nonstrict sense.) This follows plainly from the corresponding properties of functions of bounded variation.
We can thus assume without loss of generality that the function $g\in\mathrm{NBV}$ is real-valued and increasing on $[a,b]$.

Now we give the proof of Theorem~\ref{th_approx} in the form of two lemmas.

\begin{lemma}\label{lem_approx_nondecreasing}
For every increasing continuous function $g\in \mathrm{NBV}$, there exists a sequence $(g_k)_{k=1}^\infty$ of increasing functions $g_k\in S$ that satisfies conditions \eqref{MSp.granum} and \eqref{MSp.ii}.
\end{lemma}

\begin{proof}
Let a function $g\in \mathrm{NBV}$ be increasing and continuous. Choosing an integer $k\geq1$ arbitrarily, we put
\begin{equation}\label{4a.part}
t_{k,s} := \min \left\lbrace t\in[a,b] : g(t) = \frac{s g(b)}{k+1}\right\rbrace \quad \mbox{for each}\quad s\in \mathbb{Z}\cap[1,k].
\end{equation}
Considering the partition of $[a,b]$ with the points \eqref{4a.part}, we build the following increasing step function:
$$
g_k(t) :=
	\left\lbrace
	\begin{array}{cll}
		0 & \mbox{if}\quad t\in [a,t_{k,1}],\smallskip\\
\frac{sg(b)}{k+1} & \mbox{if}\quad t\in (t_{k,s}, t_{k,s+1}]\;\;
\mbox{for some}\;\; s\in\mathbb{Z}\cap[1,k-1],\smallskip\\
\frac{k g(b)}{k+1} & \mbox{if}\quad t\in (t_{k,k}, b].
	\end{array}
	\right.
$$
We see that
$$
\operatorname{V}(g_k,[a,b]) = g_k(b) < g(b)
$$
and that
$$
\sup_{a\leq t \leq b}|g_{k}(t)-g(t)|=\frac{g(b)}{k+1}\rightarrow0 \quad \mbox{as}\quad k\to\infty,
$$
which is what had to be proved.
\end{proof}

\begin{lemma}\label{lem_approx_jumps}
For every increasing jump function $g\in \mathrm{NBV}$, there exists a sequence $(g_k)_{k=1}^\infty$ of increasing jump functions $g_k\in S$ that satisfies conditions \eqref{MSp.granum} and \eqref{MSp.ii}.
\end{lemma}

\begin{proof}
 Let $g\in \mathrm{NBV}$ be an increasing jump function. This function admits the representation
$$
g(t) = \sum\limits_{s\in\omega\colon t_s<t} h_s \quad\mbox{whenever}\quad a\leq t\leq b,
$$
with $\{t_s| s\in\omega\}$ denoting the finite or countable sequence of points of discontinuity of $g$ and with $h_s > 0$ denoting the jump of $g$ at the point $t_s$.

If the set $\omega$ is finite, then $g\in S$, and the sequence $g_k:= g$, with $k\geq1$, is required. Examine now the case when $\omega$ is infinite but countable, i.e. $\omega=\mathbb{Z}_+$. Given an integer $k\geq1$, we consider the increasing jump function
$$
g_k(t):=\sum_{1\leq s\leq k,\; t_s<t} h_s \quad\mbox{whenever}\quad a\leq t\leq b.
$$
Then
$$
\sup_{a\leq t \leq b}|g_{k}(t)-g(t)|\leq  \sup_{a\leq t \leq b}\sum_{s> k,\; t_s<t} h_s \leq\sum_{s=k+1}^\infty h_s\rightarrow 0 \quad \mbox{as}\quad k\rightarrow\infty,
$$
which gives \eqref{MSp.granum}.
Condition \eqref{MSp.ii} is satisfied because $0 \leq g_k(t) \leq g(t) \leq g(b)$ for every $t\in[a,b]$.
\end{proof}

Now Theorem~\ref{th_approx} is a direct consequence of Lemmas \ref{lem_approx_nondecreasing} and \ref{lem_approx_jumps} because every increasing function $g\in\mathrm{NBV}$ can be represented as a sum of an increasing continuous function and an increasing jump function each of which belongs to  $\mathrm{NBV}$.

\section{Proofs of the main results}\label{sec_proof_t1}

\begin{proof}[Proof of Theorem\textup{~\ref{MSp.2apr_thm}}]

We will build a sequence of multipoint boundary-value problems \eqref{3MSp.syste}, \eqref{MSp.2B_m} such that
\begin{equation}\label{4aut.eq.1}
    \mathcal{X} \ni A_{r-l,k}\to A_{r-l}\quad\mbox{in}\quad (C^{(n-r)})^{m\times m}\quad\mbox{as}\quad k\to\infty\quad\mbox{whenever}\quad 1\leq l\leq r
\end{equation}
and that
\begin{equation}\label{4aut.eq.2}
    B_ky\to By\quad\mbox{in}\quad \mathbb{C}^{rm} \quad \mbox{for every}\quad y\in(C^{(n)})^{m}.
\end{equation}
According to \cite[p.~377, Main Theorem]{MikhailetsMurachSoldatov2016MFAT}, it follows from \eqref{4aut.eq.1} and \eqref{4aut.eq.2} that the problems \eqref{3MSp.syste}, \eqref{MSp.2B_m} are uniquely solvable whenever $k\gg1$ and that their solutions $y_k$ satisfy \eqref{MSp.2granum}.

Since the set $\mathcal{X}$ is dense in $(C^{(n-r)})^{m\times m}$ by our assumption, we choose sequences $(A_{r-l,k})_{k=1}^{\infty}$, with $1\leq l\leq r$, that satisfy \eqref{4aut.eq.1}. To build the required operators $B_k$, we use the unique representation of the continuous linear operator \eqref{2MSp.oper} in the form
\begin{equation*}
By=\sum_{l=0}^{n-1}\alpha_{l}\,y^{(l)}(a)+
\int\limits_a^b(d G(t))y^{(n)}(t)\quad\mbox{for every}\quad
y\in(C^{(n)})^{m}.
\end{equation*}
Here, all $\alpha_l \in \mathbb C^{rm \times m}$ are numeric matrices, and
$$
G(t)=(g^{\lambda,\mu}(t))_{\substack{\lambda=1,\ldots,rm \\\mu=1,\ldots,m}}
$$
is an $rm \times m$-matrix-valued function  such that all $g^{\lambda,\mu}\in\mathrm{NBV}$. (Certainly, the integral is understood in the Riemann--Stieltjes sense.)
This representation follows directly from the known description of the dual of  $C^{(r-1)}$  (see, e.g., \cite[Chapter~IV, Section~13, Exercise~36]{DanfordShvarts1958}). Given $\lambda$ and $\mu$, we define the functional $\ell_{\lambda,\mu}\in C^\prime$ by formula \eqref{l} in which we take $g(t)\equiv g^{\lambda,\mu}(t)$. According to Theorem~\ref{th_dense} there exists a sequence $(g_{k}^{\lambda,\mu})^\infty_{k=1}\subset S$ that satisfies \eqref{4aut.t3} in the case where $g_k=g_{k}^{\lambda,\mu}$. Hence,
\begin{equation*}
\lim_{k\to \infty}\int\limits_a^bz^{(n)}(t)dg_k^{\lambda,\mu}(t)= \int\limits_a^b z^{(n)}(t)dg^{\lambda,\mu}(t) \quad \mbox{for every}
\quad z\in C^{(n)}.
\end{equation*}
Since each $g_{k}^{\lambda,\mu}\in S$, the integral on the left is a linear combination of values of the function $z^{(n)}(t)$ at some points of $[a,b]$. Now, putting
\begin{equation*}
    B_ky:=\sum_{l=0}^{n-1} \alpha_{l}
y^{(l)}(a)+\int\limits_a^b\left(d (g_k^{\lambda,\mu}(t))_{\substack{\lambda=1,\ldots,rm \\\mu=1,\ldots,m}}\right)y^{(n)}(t)
\end{equation*}
for every $y\in (C^{(n)})^{m}$ and whenever $k\in\mathbb{N}$, 
we see that the continuous linear operators $B_k:(C^{(n)})^{m}\rightarrow \mathbb C^{rm}$ take the form \eqref{MSp.2B_m} and satisfies the required property~\eqref{4aut.eq.2}. It remains to note, that the functions $g_k^{\lambda,\mu}(t)$ and, hence, the operators $B_k$ do not depend on $f$ and $q$ and are build explicitly, as was shown in the proofs of Lemmas \ref{lem_approx_nondecreasing} and \ref{lem_approx_jumps}.
\end{proof}

Let \eqref{3MSp.syste}, \eqref{MSp.2B_m} be the multipoint boundary-value problems from Theorem~\ref{MSp.2apr_thm}. They have an additional approximation property with respect to the problem \eqref{2MSp.syste}, \eqref{2MSp.kue}.

\begin{remark}\label{4uat.rem_1}
Given $k\gg1$, consider the (uniquely solvable) boundary-value problems
\begin{align*}
L_{k}\widehat{y}_k(t)&=f_k(t)\quad\mbox{whenever}\quad a\leq t\leq b,\\
B_{k}\widehat{y}_k&=q_k,
\end{align*}
whose right-hand sides $f_k\in(C^{(n-r)})^{m}$ and $q_k\in\mathbb{C}^{rm}$ satisfy
\begin{equation}\label{4aut.r.s.est}
\|f-f_k\|_{(n-r)}<\varepsilon\quad\mbox{and}\quad
\|q-q_k\|<\varepsilon
\end{equation}
for a certain number $\varepsilon>0$, with $\|\cdot\|$ denoting the norm in $\mathbb{C}^{rm}$. Then there exist positive numbers $\varkappa$ and $\varrho$ such that
\begin{equation}\label{error}
\|y-\widehat{y}_k\|_{(n)}<\varkappa\,\varepsilon
\quad\mbox{whenever}\quad k\geq\varrho.
\end{equation}
The number $\varkappa$ can be chosen not depending on $\varepsilon$, $f$, $f_k$, $q$, and $q_k$, whereas $\varrho$ can be chosen not depending on $f_k$ and $q_k$. Specifically, if $f_k\to f$ in $(C^{(n-r)})^{m}$ and $q_k\to q$ in $\mathbb{C}^{rm}$ as $k\to\infty$, then $\widehat{y}_k\to y$ in $(C^{(n)})^{m}$ as $k\to\infty$.
\end{remark}

\begin{remark}\label{rem1}
The form of the continuous linear operators $B_k:(C^{(n)})^{m}\to\mathbb{C}^{rm}$ is not important in Remark~\ref{4uat.rem_1}. It is enough to know that the sequence of the corresponding boundary-value problems \eqref{3MSp.syste}, \eqref{MSp.2B_m} satisfies the conclusion of Theorem~\ref{MSp.2apr_thm}.
\end{remark}

Remark~\ref{rem1} follows evidently from the proof of Remark~\ref{4uat.rem_1}.

\begin{proof}[Proof of Remark~$\ref{4uat.rem_1}$]
With problems \eqref{2MSp.syste}, \eqref{2MSp.kue}  and \eqref{3MSp.syste}, \eqref{MSp.2B_m}, we associate the continuous linear operators $(L,B)$ and $(L_k,B_k)$ from $(C^{(n)})^{m}$ to $(C^{(n-r)})^{m}\times\mathbb{C}^{rm}$. The operators $(L,B)$ and $(L_k,B_k)$ whenever $k\geq\varrho_1$ are invertible by our assumption and Theorem~\ref{MSp.2apr_thm}. Here, $\varrho_1$ is some positive integer that depends only on the sequence $((L_k,B_k))_{k=1}^{\infty}$.
Consider the inverses $(L,B)^{-1}$ and $(L_k,B_k)^{-1}$ of these operators. According to Theorem~\ref{MSp.2apr_thm},
\begin{equation}\label{f16}
(L_k,B_k)^{-1}(f,q)=y_{k}\to y=(L,B)^{-1}(f,q)
\quad\mbox{in}\quad(C^{(n)})^{m}\quad\mbox{as}\quad k\to\infty
\end{equation}
for all $f\in(C^{(n-r)})^{m}$ and $q\in\mathbb{C}^{rm}$; i.e,   $(L_k,B_k)^{-1}$ converges to $(L,B)^{-1}$ in the strong operator topology. Therefore in view of the Banach-Steinhaus theorem, there exists a number $\varkappa_1>0$ such that  $\|(L_k,B_k)^{-1}\|\leq\varkappa_1$ whenever $k\geq\varrho_1$. Hence,
\begin{align*}
\|\widehat{y}_{k}-y\|_{(n)}&=
\|(L_k,B_k)^{-1}(f_k,q_k)-(L,B)^{-1}(f,q)\|_{(n)}\\
&\leq\|(L_k,B_k)^{-1}(f_k,q_k)-(L_k,B_k)^{-1}(f,q)\|_{(n)}\\
&\quad+\|(L_k,B_k)^{-1}(f,q)-(L,B)^{-1}(f,q)\|_{(n)}\\
&\leq\varkappa_{1}(\|f_k-f\|_{(n-r)}+\|q_k-q\|)+\|y_{k}-y\|_{(n)}\\
&<2\varkappa_{1}\varepsilon+\varepsilon=\varkappa\varepsilon
\quad\mbox{whenever}\quad k\geq\varrho,
\end{align*}
due to \eqref{4aut.r.s.est} and \eqref{f16}. Here, $\varkappa:=2\varkappa_{1}+1$, and the number $\varrho\geq\varrho_1$ satisfies the implication $k\geq\varrho\Rightarrow\|y_{k}-y\|_{(n)}<\varepsilon$.
\end{proof}

\medskip

{\small

\end{document}